\documentclass[11pt, a4paper]{amsart}
\usepackage[hmargin=32mm, vmargin=27mm, includefoot]{geometry}
\usepackage[colorlinks,bookmarksopen=true]{hyperref}
\usepackage[english]{babel}
\usepackage{amssymb}
\usepackage{latexsym}
\usepackage{mathrsfs}
\usepackage{xspace}
\usepackage[shortlabels]{enumitem}
\usepackage[all, cmtip]{xy}
\usepackage[usenames,dvipsnames]{color}
\usepackage[normalem]{ulem}
\newtheorem{prop}{Proposition}
\newtheorem{thm}[prop]{Theorem}
\newtheorem*{thm*}{Theorem}

\newtheorem*{addendum*}{Addendum}
\newtheorem{cor}[prop]{Corollary}
\newtheorem{lem}[prop]{Lemma}

\newtheorem*{convention*}{Convention}
\theoremstyle{definition}
\newtheorem*{defn*}{Definition}

\newtheorem*{scholium*}{Scholium}
\theoremstyle{remark}
\newtheorem{example}[prop]{Example}
\newtheorem*{example*}{Example}
\numberwithin{equation}{section}

\newcommand{\QQ}{\mathbf{Q}}
\newcommand{\RR}{\mathbf{R}}

\newcommand{\ZZ}{\mathbf{Z}}

\newcommand{\SL}{\mathrm{SL}}

\newcommand{\centra}{\mathscr{Z}}

\newcommand{\se}{\subseteq}

\def\bs#1.{
              \def\temp{#1}
              \ifx\temp\empty
                   \mathcal{B}
              \else
                   \mathcal{B}(#1)
              \fi
}
\newcommand{\cat}{{\upshape CAT(0)}\xspace}
\newcommand{\tangle}[2]
{\angle_\mathrm{T}(#1,#2)}
\newcommand{\aangle}[3]
{\angle_{#1}(#2,#3)}
\newcommand{\cangle}[3]
{\overline{\angle}_{#1}(#2,#3)}

\DeclareMathOperator{\Isom}{Is}

\newcommand{\bd}{\partial} 

\def\Aut{\mathop{\mathrm{Aut}}\nolimits}

\makeindex
\begin{document}
\title[Erratum and addenda]{Erratum and addenda to\\
``Isometry groups of non-positively curved spaces: discrete subgroups''}
\author[P.-E. Caprace]{Pierre-Emmanuel Caprace*}
\address{UCL -- Math, Chemin du Cyclotron 2, 1348 Louvain-la-Neuve, Belgium}
\email{pe.caprace@uclouvain.be}
\thanks{* F.R.S.-FNRS senior research associate}
\author[N. Monod]{Nicolas Monod}
\address{EPFL, 1015 Lausanne, Switzerland}
\email{nicolas.monod@epfl.ch}
\date{August 2019}
%
\begin{abstract}
We amend the statement of point~(i) in Theorem~1.3 in~\cite{Caprace-Monod_discrete} and supply the additional arguments and minor changes for the results that depend on it. We also seize the occasion and generalize to non-finitely generated lattices.
\end{abstract}
\maketitle


\section{The Euclidean Factor Theorem for \cat lattices}

\begin{flushright}
\begin{minipage}[t]{0.55\linewidth}\itshape\small
I'm fixing a hole where the rain gets in\\
And stops my mind from wandering
\begin{flushright}
\upshape Lennon--McCartney, \emph{Fixing a Hole}, 1967
\end{flushright}
\end{minipage}
\end{flushright}

\vspace{.5cm}
Let $X$ be a proper \cat space and let $n\geq 0$ be the dimension of the maximal Euclidean factor of $X$. Let $G < \Isom(X)$ be a closed subgroup acting minimally and cocompactly on $X$, and let $\Gamma < G$ be a finitely generated lattice. Theorem~1.3(i) in~\cite{Caprace-Monod_discrete} states that $\Gamma $ has a finite index subgroup $\Gamma_0$ that splits as $\Gamma_0 \cong \ZZ^n \times \Gamma'$, and that $n$ coincides with the maximal rank of a free abelian normal subgroup of $\Gamma$.

This result is incorrect,  as shown by the example below, which is a simple case of a beautiful general construction of Leary--Minasyan~\cite{LM} for which they prove a number of deep results. Our mistake lies in the erroneous assertion from the proof of Proposition~3.6 loc.~cit.\ that ``the commensurator of any lattice in $A$ is virtually abelian''.

However, the \textit{normaliser} of any lattice in the Euclidean motion group $A = \mathbf R^k \rtimes O(k)$ is indeed virtually abelian (see Lemma~\ref{lem:normaliser} below). Relying on this, we state and prove an amended version of Theorem~1.3(i), namely Theorem~\ref{thm:EuclideanFactor} below. That result ensures similarly that $n$ coincides with the maximal rank of a free abelian \textit{commensurated} subgroup of $\Gamma$.

Concerning the virtual splitting of $\Gamma$ as $\ZZ^n \times \Gamma'$, it does nonetheless hold provided $\Gamma$ is residually finite, see Theorem~\ref{thm:EuclideanFactor}\ref{it:ResFin}.

\begin{example}[See~\cite{LM}]
The construction starts from a Pythagorean triple, say $(3,4,5)$. The matrix 
$\alpha = \left( \begin{array} {rr} 
3/5 & 4/5 \\
-4/5 & 3/5
\end{array}
\right)$
is in $\SL_2(\ZZ[1/5]) \cap SO(2)$ and represents an  irrational rotation of $\RR^2$. Therefore the corresponding semi-direct product $\Lambda = (\ZZ[1/5])^2 \rtimes_{\alpha} \ZZ$ embeds as a dense subgroup in $A=\RR^2 \rtimes SO(2)$. Moreover, $\Lambda$ is finitely generated (by $\alpha$ together with a suitable pair of independent elements of $ (\ZZ[1/5])^2$). Let now $D' = (\QQ_5)^2 \rtimes_{\alpha} \ZZ$ be the semi-direct product of the additive group of the $2$-dimensional vector space over the $5$-adic rationals, by the cyclic  automorphism group generated by $\alpha$. The diagonal embedding of $(\ZZ[1/5])^2$ in $\RR^2 \times (\QQ_5)^2$ is a cocompact lattice embedding. It follows that $\Lambda$ embeds diagonally as a cocompact lattice in $A \times D'$, with dense projection on both factors. Since $D'$ is compactly generated (because it contains a dense copy of $\Lambda$) and does not have non-trivial compact normal subgroup, it has a continuous 
 faithful vertex-transitive action on a locally finite graph $\mathfrak g$ (a ``Cayley--Abels graph'', see e.g.\ the end of~5.2 in~\cite{Abels74}). Let $D$ be the extension of $D'$ by the fundamental group of $\mathfrak g$. Then $D$ acts vertex-transitively on universal cover of $\mathfrak g$, which is a regular locally finite tree $T$. Moreover the cocompact irreducible lattice $\Lambda< A \times D'$ lifts to a cocompact irreducible lattice $\Gamma$ in $G = A \times D$.  Notice that $G$ acts cocompactly and minimally on the proper \cat space $X$ defined by $X = \RR^2 \times T$. In particular $\Gamma$ is a \cat group.

On the other hand, no finite index subgroup of $\Gamma$ has a finitely generated free abelian normal subgroup. Indeed, we first observe that the projection $p(N)$ to $A$ of such a subgroup $N < \Gamma$ must be trivial. To prove this claim, notice that by construction, the projection of $\Gamma$ to $A$ is the metabelian group $\Lambda$. The intersection of $p(N)$ with the translation subgroup 
%
%
$ V= \ZZ[1/5]^2$ of $\Lambda$ is a finitely generated free abelian normal subgroup. Any finitely generated subgroup of $V$ is generated by at most two elements, and is thus a discrete subgroup of $\RR^2$. Since   $p(N) \cap V$ is normalised by the irrational rotation   $\alpha$, it follows that $p(N) \cap V$  is trivial. Therefore $p(N)$ commutes with  $V$. But $V$ is its own centraliser in $\Lambda$, hence $p(N) < V$ and the claim follows since $p(N) \cap V$ is trivial.
%
%
It follows that every finitely generated free Abelian subgroup  of $\Gamma$ would be an abelian normal subgroup of $D$, and hence it is necessarily trivial by~\cite[Theorem~1.10]{Caprace-Monod_structure}.
\end{example}

Observe that $G$ is a compactly generated locally compact group whose amenable radical coincides with $A$. The example shows moreover that the image of a cocompact lattice in $G$  need not project to a lattice in the quotient of $G$ by its amenable radical. This lies in sharp contrast with the case of connected Lie groups, where the image of a lattice modulo the amenable radical is always a lattice, by Auslander's theorem~\cite[Theorem~23.9.3]{Raghunathan}. 

\medskip
Theorem~1.3(i) from~\cite{Caprace-Monod_discrete} should be replaced by the following; notice that no finite generation assumption is made in~\ref{it:commen}.

\begin{thm}\label{thm:EuclideanFactor}
Let $X$ be a proper \cat space, $G< \Isom(X)$ a closed subgroup acting minimally and cocompactly, and $\Gamma < G$ any lattice. Let $E\cong\RR^n$ be the maximal Euclidean factor of $X$, where $n\geq 0$.

\begin{enumerate}[(i), leftmargin=2em]
\item\label{it:commen} $\Gamma$ commensurates a free abelian subgroup $\Gamma_A\cong \ZZ^n$, and $n$ is the largest such rank.\\
Moreover, any commensurated abelian subgroup of $\Gamma$ acts properly on $E$.
\end{enumerate}

\noindent
We now assume $\Gamma$ finitely generated.

\begin{enumerate}[resume*]
\item \label{it:norma} If $\Gamma$ virtually normalises a free abelian subgroup of rank~$k$, then  $\Gamma $ virtually splits as $\ZZ^k \times \Gamma'$. Moreover there is a corresponding invariant decomposition $X \cong \RR^{k} \times X'$, and the projection of $\ZZ^k$ (resp. $\Gamma'$) to $\Isom(X')$ is trivial (resp. discrete).

\item \label{it:VirtAb} If the projection of $\Gamma$ to $\Isom(E)$ is virtually abelian, then $\Gamma $ virtually splits as $\ZZ^n \times \Gamma'$.

\item \label{it:ResFin} If $\Gamma$ is residually finite, then again $\Gamma $ virtually splits as $\ZZ^n \times \Gamma'$.

\item  \label{last-item} If $\Gamma$ is cocompact, then  $\Gamma $ virtually splits as $\ZZ^k \times \Gamma'$ and there is a    corresponding invariant decomposition $X \cong \RR^{k} \times X'$ for some $k \geq 0$, satisfying the same  properties as in \ref{it:norma},  and  such that moreover every $\Gamma'$-orbit in $\partial X'$ is infinite.

\end{enumerate}
\end{thm}

We recall that the \emph{virtual splittings} above mean by definition that a finite index subgroup splits as specified. Item~\ref{last-item} implies  that if $\Gamma$ fixes a point in $\partial X$, then $\Gamma$ virtually splits as $\ZZ\times \Gamma'$.

\medskip

The condition that $\Gamma$ normalizes a free abelian subgroup in~\ref{it:norma} is natural, since the center of $\Gamma$, and more generally the amenable radical of $\Gamma$, is virtually free abelian. In particular, we obtain the following.

\begin{cor}\label{cor:Ramen}
Let $X$ be a proper \cat space whose maximal Euclidean factor has dimension $n$. Let $G< \Isom(X)$ be a closed subgroup acting minimally and cocompactly, and $\Gamma < G$ any lattice. 

Then the amenable radical of $\Gamma$ is virtually isomorphic to $\mathbf Z^k$ for some $k \leq n$. If in addition $\Gamma$ is finitely generated, then $\Gamma $ has a finite index subgroup $\Gamma_0$ that splits as $\Gamma_0 \cong \ZZ^k \times \Gamma'$, where $k \leq n$ and $\Gamma'$ has trivial amenable radical (hence a  trivial centre).

Moreover, there is a corresponding $\Gamma_0$-invariant decomposition $X \cong \RR^{k} \times X'$ and the projection of $\Gamma_0$ to $\Isom(X')$ is discrete.
\end{cor}

\medskip

The proof of Theorem~\ref{thm:EuclideanFactor}\ref{it:commen} requires the following fact, which supplements a result on normal subgroups proved in~\cite[Theorem~1.10]{Caprace-Monod_structure}. 

\begin{prop}\label{prop:commensurated}
Let $X \neq \RR$ be an irreducible proper \cat space with finite-dimensional Tits boundary and $G < \Isom(X)$ any subgroup whose action is minimal and does not have a global fixed point in $\bd X$. 

Then any commensurated subgroup $H < G$ either fixes a point in $X$, or still acts minimally.
\end{prop}

\begin{proof}
Let $\Delta H$ be the \textbf{convex limit set} of $H$, defined as the visual boundary $\bd Y$ of the closed convex hull $Y$ of some $H$-orbit. By~\cite[Lemma~4.2]{Caprace-Monod_structure}, this set is well-defined, i.e. it does not depend of the choice of a specific $H$-orbit. If $H_0<H$ is a finite index subgroup, then each $H_0$-orbit is cobounded in the corresponding $H$-orbit and it follows that its convex hull is cobounded in the convex hull of the $H$-orbit. Thus, these two convex sets have the same visual boundary. It follows that if $H'$ is commensurable to $H$ in $G$, then $\Delta H' = \Delta H$. Since $H$ is commensurated by $G$ by hypothesis, the closed convex subset $\Delta H \se \bd X$ is $G$-invariant. If $\Delta H$ is empty, then $H$ has bounded orbits and fixes points in $X$. If $\Delta H$ is non-empty, then it must have circumradius $> \pi/2$ since otherwise $G$ fixes a point in $\bd X$ by~\cite[Proposition~1.4]{BalserLytchak_Centers}. By~\cite[Proposition~3.6]{Caprace-Monod_structure} the union of all minimal closed convex subspaces $Z \se X$ with $\bd Z = \Delta H$ is a non-empty closed convex subspace which splits as a product $Z_0 \cong Z \times Z'$. Since $\Delta H$ is $G$-invariant, so is $Z_0$. By hypothesis $X$ is irreducible and the $G$-action is minimal. It follows that $X = Z_0 = Z$ and that $\Delta H = \bd X$. In particular for every closed convex subset $X' \subsetneq X$, we have $\bd X' \subsetneq \bd X$, and we infer that the convex hull of every $H$-orbit is dense in $X$. In other words $H$ acts minimally on $X$. 
\end{proof}

As mentioned above, we shall also need the following. 
%
%
\begin{lem}\label{lem:normaliser}
Let $G = \mathbf R^k \rtimes O(k)$ and $\Gamma \leq G$ be a lattice. Then the normaliser $N_G(\Gamma)$ has a finite index subgroup contained in the translation group $\mathbf R^k$. In particular 	$N_G(\Gamma)$ is virtually abelian. 
\end{lem}
\begin{proof}
Let $V = \mathbf R^k$ be the translation subgroup of $G$ and set $V_\Gamma = V \cap \Gamma$. Since $V$ is normal in $G$, we have $N_G(\Gamma) \leq N_G(V_\Gamma) =: H$. The group $V_\Gamma$ is a discrete normal subgroup of $H$. Therefore every element of $V_\Gamma$ has a discrete conjugacy class in $H$, so that the centraliser $C_H(v)$ is open in $H$ for all $v \in V_\Gamma$. By the Bieberbach Theorem, $V_\Gamma$ is a lattice in $V$, hence it is finitely generated. It follows that $C_H(V_\Gamma)$ is open in $H$. Clearly $V \leq C_H(V_\Gamma) \leq C_G(V_\Gamma) \leq N_G(V_\Gamma) = H$. On the other hand the centraliser in $G$ of the lattice $V_\Gamma$ must act trivially at infinity of the Euclidean space, so that $C_G(V_\Gamma) \leq V$. Therefore $V = C_H(V_\Gamma) = C_G(V_\Gamma)$, and the natural  image of $H$ in $G/V \cong O(k)$ is closed (because $V \leq H$) and at most countable (because $V$ is open in $H$). Since $O(k)$ is compact, the natural image of $H$ in $O(k)$ is thus finite by Baire's theorem. In particular, so is the image of $N_G(\Gamma)$, as required. 
\end{proof}	

\begin{lem}\label{lem:SAD}
Let $A = \RR^n \rtimes O(n)$, $S$ be a semi-simple Lie group with trivial centre and no compact factor, $D$ be a totally disconnected locally compact group, and $G = S \times A \times D$.

Then any lattice $\Gamma < G$  commensurates a free abelian subgroup $\Gamma_A < \Gamma$ of rank $n$. Moreover the  projection of $\Gamma_A$ to $A$ (resp. to $S$) is a lattice (resp. is trivial). 
\end{lem}

\begin{proof}
 Let $Q < D$ be a compact open subgroup. By~\cite[Lemma~3.2]{Caprace-Monod_discrete}, the intersection $\Gamma^* = \Gamma \cap (S \times A \times Q)$ is a lattice in $S \times A \times Q$, which is commensurated by $\Gamma$. Since $Q$ is compact, the projection of $\Gamma^*$ to $S \times A$ is a lattice. Upon replacing $\Gamma^*$ by a finite index subgroup, we may then assume by Lemmas~3.4 and~3.5 from~\cite{Caprace-Monod_discrete} that $\Gamma^*$ possesses two normal subgroups $\Gamma^*_S$ and $\Gamma^*_A$, both commensurated by $\Gamma$, such that $\Gamma^* = \Gamma^*_S \cdot \Gamma^*_A$ and $\Gamma^*_S \cap \Gamma^*_A \se Q$, where $\Gamma^*_A = \Gamma^* \cap (1 \times A \times Q)$ is a finitely generated free abelian group whose projection to $A$ is a lattice. In particular $\Gamma^*_A$ is a free abelian group of rank $n$ which is commensurated by $\Gamma$ and we can define $\Gamma_A=\Gamma^*_A$.
\end{proof}

\begin{proof}[Proof of Theorem~\ref{thm:EuclideanFactor}]
By~\cite[Theorem~M]{amenis}, the only points of $\bd X$ that are possibly fixed by $G$ lie in the visual boundary of the maximal Euclidean factor of $X$. In particular the full isometry group $\Isom(X)$ has no fixed point in $\bd X$. Theorem~1.6 and Addendum~1.8 from~\cite{Caprace-Monod_structure} then provide a canonical $\Isom(X)$-invariant decomposition $X \cong M \times \RR^n \times Y$, where $\RR^n$ is the maximal Euclidean factor of $X$, $S = \Isom(M)$ is a semi-simple Lie group with trivial center and no compact factor, and $D = \Isom(Y)$ is totally disconnected. In particular $\Gamma$ is a lattice in $\Isom(X) = S \times A \times D$, where $A \cong \RR^n \rtimes O(n)$. 

The fact that   $\Gamma$ contains a commensurated free abelian subgroup of rank $n$ now follows from Lemma~\ref{lem:SAD}.
Assume conversely that $\Gamma$ commensurates some abelian subgroup $H$; we only need to consider the case where $H$ is infinite. By Borel density, the projection of $H$ to $S$ is finite; we may thus assume that it is trivial upon replacing $H$ by a finite index subgroup. Recall that an abelian group cannot act minimally on a non-trivial proper \cat space without Euclidean factor (indeed, since the displacement function of an element $g$ is invariant under the centraliser $\centra_H(g)$, the minimality of $H$ implies that $g$ has constant displacement function, and is thus a Clifford isometry). Therefore, Proposition~\ref{prop:commensurated} implies that $H$ fixes a point in each irreducible factor of $Y$, hence also in $Y$ itself. In other words the closure of the projection of $H$ to $S \times D$ is compact. It follows that the $H$-action on the maximal Euclidean factor $\RR^n$ is indeed proper. In particular, if $H$ is free abelian of rank $m$, then $m \leq n$. This proves~\ref{it:commen}. 

\medskip
For the rest of the proof, we suppose that $\Gamma$ is finitely generated.

\medskip
Assume that $\Gamma$ normalises a free abelian subgroup of rank $k$, say $H$. By~\ref{it:commen}, the group $H$ acts properly on the maximal Euclidean factor $\RR^n$. Since it is normal, it yields a $\Gamma$-invariant decomposition $\RR^n = \RR^k \times \RR^{n-k}$, and the $H$-action on the $\RR^{n-k}$-factor is trivial. Since the projection of $\Gamma$ to $\Isom(\RR^k)$ normalises a lattice, it is virtually abelian by Lemma~\ref{lem:normaliser}. Upon passing to a finite index subgroup, we may thus assume that the image of $\Gamma$ in $\Isom(\RR^k)$ is free abelian of the form $\ZZ^k \oplus \ZZ^m$, where $\ZZ^k$ is the image of $H$. Let $\Lambda$ be the preimage of the $\ZZ^m$-factor. Then $\Gamma = H \cdot \Lambda$. Moreover the intersection $H \cap \Lambda$ acts trivially on $\RR^k$. By construction $H$ acts trivially on the complementary Euclidean factor $\RR^{n-k}$. Since moreover $H$ is an abelian normal subgroup of a lattice, it acts trivially on $M \times Y$ by~\cite[Theorem~1.10]{Caprace-Monod_structure} and~\cite[Theorem~1.1]{Caprace-Monod_discrete}. Therefore $H \cap \Lambda$ is trivial, whence $\Gamma \cong H \times \Lambda$. Moreover the projection of $H$ to $\Isom(\RR^{n-k}) \times S \times D$ is trivial, hence the projection of $\Gamma$ is discrete by~\cite[Proposition~3.1]{Caprace-Monod_discrete}. This proves~\ref{it:norma}.

\medskip

Assume now that $\Gamma$ projects to a virtually abelian subgroup of $A$. Since $\Gamma$ is finitely generated, upon replacing $\Gamma$ by a finite index subgroup, we may thus assume that the image of $\Gamma$ is free abelian of the form $\ZZ^n \oplus \ZZ^m$, where $\ZZ^n$ is the image of the commensurated subgroup $\Gamma_A$ provided by~\ref{it:commen}, which acts thus properly on $\RR^n$ and fixes a point in $M \times Y$. Let $\Lambda$ be the preimage of the   $\ZZ^m$-factor in $\Gamma$. Thus $\Lambda$ is normal in $\Gamma$ and we have  $\Gamma = \Gamma_A \cdot \Lambda$. 

Since the closure of the projection of $\Gamma_A$ to $S \times D$ is compact, it follows that the closure of the projection of $\Lambda$, which is denoted by $L$, is cocompact in the closure of the projection of $\Gamma$, which is denoted by $J$. Since $J$ is compactly generated (because $\Gamma$ is finitely generated and maps densely to $J$), the cocompact subgroup $L$ is compactly generated as well. Moreover the derived group $[\Lambda, \Lambda]$ maps trivially to $A$ and is thus a discrete normal subgroup of $L$. It follows that the quotient $L/[\Lambda, \Lambda]$ is a compactly generated abelian locally compact group. It is therefore of the form $K \times \RR^p \times \ZZ^q$ for some $p, q \geq 0$, where $K$ is a compact subgroup (see~\S29 in~\cite{Weil} or~9.8 in~\cite{HR}). The identity component of $L$ maps trivially to $D$. It is thus a connected subgroup of $S$ normalised by the projection of a lattice. By Borel density, it follows that $L^\circ$ is semi-simple, and can therefore not have any non-trivial connected abelian quotient. In view of Lemma~2.4 from~\cite{CCMT}, it follows that  $p=0$, i.e.  $L/[\Lambda, \Lambda] \cong K \times  \ZZ^q$. Now the preimage of $\ZZ^q$ in $L$ is a cocompact discrete normal subgroup. It must therefore be finitely generated, and its centraliser is thus an open subgroup of $L$. By~\cite[Theorem~1.1]{Caprace-Monod_discrete}, the group $L$ acts minimally without a fixed point at infinity on $M \times Y$, and any  cocompact subgroup of $L$ must therefore have trivial centraliser.  It follows that $L$ is discrete. Thus $L$ is a discrete cocompact normal subgroup of $J$. By the same arguments we have just used, it follows that $L$ is finitely generated and has an open centraliser in $J$, from which it follows that $J$ is discrete as well. 

This now implies that, upon replacing $\Gamma_A$ by a finite index subgroup, the projection of $\Gamma_A$ to $S\times D$ is trivial. It then follows that $\Gamma_A$ is normal in $\Gamma$, and  assertion~\ref{it:VirtAb} therefore follows from~\ref{it:norma}. 

\medskip
We now turn to~\ref{it:ResFin}. Let $\Lambda<\Gamma$ be the profinite closure of $\Gamma_A$ in $\Gamma$, that is, the smallest separable subgroup of $\Gamma$ containing $\Gamma_A$. Then $\Lambda$ contains a finite index subgroup $\Lambda_0<\Lambda$ that is normal in $\Gamma$. Indeed, this follows from the main result of~\cite{CKRW} since $\Gamma$ is finitely generated (or it could be deduced from~\cite[Cor.~4.1]{Caprace-Monod_dec}; see~\cite[Rem.~5]{CKRW}).

Since we now assume that $\Gamma$ is residually finite, its profinite topology is Hausdorff and hence $\Lambda$ is commutative. In particular, $\Lambda_0$ acts properly on $E$ by~\ref{it:commen}; it is thus finitely generated and hence, upon possibly replacing it by a further finite index subgroup that is normal in $\Gamma$, we have $\Lambda_0\cong \ZZ^k$ with $k\leq n$. Recalling that $\Lambda$ contains $\Gamma_A$, we have $k=n$ and now the conclusion follows from~\ref{it:norma}.

\medskip
It remains to prove~\ref{last-item}. Let $\Gamma_1\leq \Gamma$ be a finite index subgroup. By \cite[Th.~3.14]{Caprace-Monod_discrete}, the group $\Gamma_1$ acts minimally on $X$. By \cite[Prop.~3.15]{Caprace-Monod_discrete},   there is a $\Gamma_1$-invariant splitting $X \cong E \times X'$ such that $E$ is flat,  $\Gamma_1$ acts  trivially on $\partial E$ and has no fixed point in $\partial X'$. Among all such finite index subgroups, we now pick one, say $\Gamma_0$, for which the flat factor $E$ is of maximal possible dimension. Thus we have a $\Gamma_0$-invariant splitting $X \cong \RR^k \times X'$, and the choice of  $\Gamma_0$ ensures that every $\Gamma_0$-orbit in $\partial X'$ is infinite.  

Fix $k$ pairs $(\xi_1, \xi'_1), \dots, (\xi_k, \xi'_k)$ of antipodal pairs in  the boundary of the flat factor $\RR^k$,  so that the Tits distance from $\xi_i$ to $\xi_j$ is $\pi/2$ for all $i \neq j$. Invoking \cite[Prop.~K]{amenis} $k$ times successively, we deduce that the kernel $\Gamma'$ of the projection of $\Gamma_0$ to $\Isom(\RR^k)$ acts cocompactly on $X'$. Denoting by $H$   the closure of the projection of $\Gamma_0$ to $\Isom(X')$, we infer that $\Gamma'$ is a normal  cocompact lattice in $H$. Since $\Gamma_0$ is finitely generated, the group $H$ is compactly generated, so that $\Gamma'$ is finitely generated. The centralizer $\centra_H(\Gamma')$ is thus open in $H$, so that $\centra(\Gamma') = \Gamma' \cap  \centra_H(\Gamma')$ is a cocompact lattice in $\centra_H(\Gamma')$. By~\cite[Th.~1.10]{Caprace-Monod_structure} and~\cite[Cor.~2.7]{Caprace-Monod_discrete},  the centralizer $\centra_H(\Gamma')$ acts properly on the maximal Euclidean flat factor of $X'$. Therefore the discrete group $\centra(\Gamma')$ also acts properly on a flat, and is thus finitely generated.  If $\centra(\Gamma')$ were infinite, then $\Gamma_0$ would normalize a free abelian subgroup of positive rank contained in $\Gamma'$. In view of \ref{it:norma}, this would yield a finite $\Gamma_0$-orbit in $\partial X'$, which is impossible. Therefore $\centra(\Gamma')$ is finite, so that $\centra_H(\Gamma')$ is compact (because it contains $\centra(\Gamma')$ as a lattice). Since $\centra_H(\Gamma')$ is normal in $H$ and $H$ acts minimally on $X'$, we deduce that $H$ is discrete. This implies that the kernel, denoted by $V$, of the projection $\Gamma_0 \to \Isom(X')$ is a lattice in $\Isom(\RR^k)$. Thus $V$ is a normal subgroup of $\Gamma_0$ virtually isomorphic to $\ZZ^k$. The required conclusion now follows from  \ref{it:norma}. 
\end{proof}

\begin{proof}[{Proof of Corollary~\ref{cor:Ramen}}]
Retaining the notation of the proof of Theorem~\ref{thm:EuclideanFactor}, we know from \cite[Theorem~1.10]{Caprace-Monod_structure} and~\cite[Theorem~1.1]{Caprace-Monod_discrete} that the amenable radical of $\Gamma$ acts trivially on $M \times Y$. Therefore it is a discrete group acting properly on the flat factor $\mathbf R^n$, and is consequently virtually isomorphic to $\mathbf Z^k$. If $\Gamma$ is finitely generated, we may invoke  Theorem~\ref{thm:EuclideanFactor}(ii) and the required conclusions follow. 	
\end{proof}

We end this section with an example showing that the finite generation hypothesis cannot be discarded in Theorem~\ref{thm:EuclideanFactor}(ii).

\begin{example}\label{exam:RxT}
  We construct a CAT($0$) lattice $\Gamma$ with a normal subgroup isomorphic to $\ZZ$, but no finite index subgroup splitting as a direct product $\ZZ\times \Gamma'$, as follows.

Let $A = \ZZ[\frac 1 2]$ and let $B = \ZZ[\frac 1 2]/\ZZ$ be the the Pr\"ufer $2$-group. Then $A$ is a non-trivial central extension of $B$. Let $B^{*3} =B*B*B$ and consider the retraction $\pi\colon B^{*3} \to B$ onto the first factor. We define $\Gamma$ to be the pull-back by $\pi$ of the central extension $A$, thus containing a central subgroup $\ZZ$. We refer to e.g.~\cite[p.~94]{Brown} for the pull-back of an extension and highlight just two facts. First, $\Gamma$ is a fibered product $A \times_B B^{*3}$ over $B$ and hence, in particular, $\Gamma$ is a subgroup of $A\times B^{*3}$. Secondly, the central extension $\Gamma$ is non-trivial because it is classified by the inflation map $\pi^*\colon H^2(B, \ZZ) \to H^2(B^{*3}, \ZZ)$ and the latter is injective since $\pi$ admits a right inverse, the canonical inclusion.

We claim that $\Gamma$ does not have any finite index subgroup $\Gamma_0$ that split as $\Gamma_0 = \ZZ\times \Gamma'$. Indeed, assume for a contradiction that such a finite index subgroup $\Gamma_0$ exists. We identify $A$ with its image in $\Gamma$ and   set $A_0 = A \cap \Gamma_0$. Since the center of $\Gamma$ is contained in $A$, we see that $A_0$ contains a finite index subgroup of the   center of $\Gamma$. The group $A$ is $2$-divisible (i.e. for every $n$, every element admits a $2^n$-th root), so every finite index subgroup is of odd index, and is thus itself $2$-divisible. In particular, the restriction of the projection map $\Gamma_0\to \ZZ$ to $A_0$ must be trivial. It follows that the restriction of the projection map $\Gamma_0\to \Gamma'$ to $A_0$ is injective. Since $B$ is both $2$-divisible and $2$-torsion, it does not admit any non-trivial finite quotient. Therefore the same holds for $B^{*3}$. Using the classification of commuting elements in free products (see e.g.~\cite[4.5]{MKS}), we see that the restriction of the projection map $\Gamma \to  B^{*3}$ to $\Gamma_0$ is trivial on the $\ZZ$ factor, and induces an isomorphism $ \Gamma' \to B^{*3}$. From the discussion above, we infer that $A_0$, which contains an infinite cyclic central subgroup of $\Gamma$,  injects in $B^{*3}$ under the projection map $\Gamma \to  B^{*3}$. This is absurd, since the center of $B^{*3}$ is trivial. This contradiction confirms the claim.

Finally, we proceed to realize $\Gamma$ as a lattice in the isometry group of the cocompact CAT($0$) space $\RR\times T$, where $T$ is the trivalent tree. The group $B^{*3}$ can be seen as the fundamental group of a graph of finite groups, consisting of three rays of groups emanating from a common origin. The vertex and edge groups attached to vertices and edges at distance $n$ from the origin are cyclic groups of order~$2^n$. The fundamental group of that graph of groups indeed acts properly on $T$, and is a non-uniform lattice in $\Aut(T)$ by Serre's covolume formula. By construction, $\Gamma$ is a subgroup of $A \times  B^{*3} \leq \Isom(\RR) \times \Aut(T)$. Since the image of the projection of $\Gamma$ to $\Aut(T)$ is a lattice in $\Aut(T)$ and the kernel of that projection is a lattice in $\Isom(\RR)$ (namely $\ZZ$), it follows that $\Gamma$ is indeed a CAT($0$) lattice, as required. 
\end{example}

\section{Further corrections}
The main point of this erratum is that Theorem~1.3(i) in~\cite{Caprace-Monod_discrete} should be replaced by Theorem~\ref{thm:EuclideanFactor} above. All other results presented throuhout the Introduction of~\cite{Caprace-Monod_discrete} remain valid without change.

We now proceed to indicate all modifications in the body of~\cite{Caprace-Monod_discrete} that are required by this correction. We only discuss the results that rely directly or indirectly on Theorem~1.3(i) as all the other results are unaffected and hold without changes, with one exception: in Section~6.C of~\cite{Caprace-Monod_discrete}, the construction of pathological examples is mistaken. Specifically, the properness asserted in Lemma~6.9 does not hold.

\begin{list}{\textbullet}{\leftmargin=1em \itemindent=0em}
\setlength{\itemsep}{5pt}
\item Theorem~1.3(ii) remains valid, but it contains a typo: the statement should read ``$\Isom(X)$ has no fixed point at infinity''. We remark that Theorem~1.3(ii) has been generalised to infinitely generated lattices in~\cite[Theorems~L and~M]{amenis}.

\item Proposition~3.6 should be replaced by Lemma~\ref{lem:SAD} above. 

\item Theorem~3.8, which is a reformulation of Theorem~1.3(i), should be replaced by Theorem~\ref{thm:EuclideanFactor} above. 

\item Corollary~3.10 should be replaced by Theorem~\ref{thm:EuclideanFactor}\ref{it:commen} above. When the lattice is residually finite, then Corollary~3.10 holds as is by Theorem~\ref{thm:EuclideanFactor}\ref{it:ResFin}.

\item Theorem~4.2 remains valid as is, although its proof currently relies on Theorem~3.8. The proof can be corrected as follows. We use the notation from loc. cit. Let $X'_2 = \RR^n \times X''_2$ be the canonical decomposition of $X'_2$, where $\RR^n$ is the maximal Euclidean factor. Let $G_2$ be the closure of the projection of $\Gamma_0$ to $\Isom(X'_2)$ and $\Gamma_2 < G_2$ be the kernel of the projection of $\Gamma_0$ to $\Isom(X'_1)$. The current proof shows that $\Gamma_2$ is a normal, hence cocompact and finitely generated, lattice in $G_2$. Moreover the centraliser $\centra_{G_2}(\Gamma_2)$ is open in $G_2$. Therefore the intersection $\Gamma_2\cap \centra_{G_2}(\Gamma_2)$ is a lattice in $\centra_{G_2}(\Gamma_2)$. Since $\Gamma_2\cap \centra_{G_2}(\Gamma_2)$ is nothing but the center of $\Gamma_2$, it is an abelian normal subgroup of $\Gamma_0$.  Since $\Gamma_0$ is irreducible by hypothesis, Theorem~\ref{thm:EuclideanFactor}\ref{it:norma} implies that $\Gamma_2\cap \centra_{G_2}(\Gamma_2)$ is finite. Therefore $\centra_{G_2}(\Gamma_2)$ is compact, hence trivial since $G_2$ acts minimally on $X'_2$. Hence the group $G_2$ is discrete.
We infer that the inclusion $\Gamma_0 < G_1 \times G_2$ is of finite index, which contradicts the irreducibility assumption on $\Gamma_0$. This finishes the first part of the proof when the $G$-action on $X$ is minimal. 

The case when the $G$-action on $X$ is not minimal reduces to the minimal case, since a $\Gamma_0$-invariant splitting $X = X_1 \times X_2$ induces a $\Gamma_0$-invariant splitting $X' = X'_1 \times X'_2$ of the canonical minimal $\Gamma_0$-invariant subspace $X'$. The argument given in~\cite{Caprace-Monod_discrete} relies on the fact that $X'$ has no Euclidean factor; however the statement holds even in the presence of a Euclidean factor, since $X'$ is co-bounded in $X$ and thus has the same Tits boundary. The desired decomposition of $X'$ can then be obtained as a consequence of Propositions~3.6 and~3.11 from~\cite{Caprace-Monod_structure}. 

Finally, it remains to prove that if $G$ acts minimally on $X$ and $\Gamma = \Gamma' \times \Gamma''$ splits non-trivially, then there is a non-trivial $\Gamma$-equivariant splitting $X = X_1 \times X_2$ such that the projection of $\Gamma$ to at least one of the two factors is discrete.

By~\cite[Th.~1.10]{Caprace-Monod_structure} and~\cite[Cor.~2.7]{Caprace-Monod_discrete}, on each irreducible non-Euclidean factor of $X$, either the projection of $\Gamma'$ is trivial or the projection of $\Gamma''$ is trivial. Thus there is a canonical decomposition $X = \RR^n \times X'_1 \times \dots \times X'_p \times X''_1 \times \dots \times X''_q$ such that $X'_i$ and $X''_j$ are non-Euclidean irreducible for all $i, j$, and moreover the $\Gamma'$-action on $X''_j$ (resp. the $\Gamma''$-action on $X'_i$) is trivial.  

Assume first that $X = \RR^n \times X'_1 \times \dots \times X'_p$. Then $\Gamma''$ is a discrete subgroup of $\Isom(\RR^n)$. Since $\Gamma$ acts cocompactly on $\RR^n$, its action is minimal, so the direct factor $\Gamma''$ must be infinite. Hence $\Gamma''$ is virtually isomorphic to $\ZZ^k$ for some $k \geq 1$, and the desired splitting of $X$ is afforded by Theorem~\ref{thm:EuclideanFactor}\ref{it:norma}. Assume now that $X''_1 \times \dots \times X''_q$ is not reduced to a single point. Let  $G_1$ be the closure of the projection of $\Gamma$ to $\Isom(\RR^n \times X'_1 \times \dots \times X'_p)$. Then $\Gamma'$ is a finitely generated discrete normal subgroup of $G_1$, so   the centraliser $\centra_{G_1}(\Gamma')$ is open in $G_1$, and acts trivially on $X'_i$ for all $i$ by \cite[Cor.~2.7]{Caprace-Monod_discrete}. Notice that   $\Gamma' \cap \centra_{G_1}(\Gamma') = \centra(\Gamma')$ is finitely generated, since it acts properly on $\RR^n$. If $\centra(\Gamma')$ is infinite, then $\Gamma$ virtually normalizes  a subgroup isomorphic to $\ZZ^k$ for some $k \geq 1$, and we conclude again by invoking Theorem~\ref{thm:EuclideanFactor}\ref{it:norma}. If $\centra(\Gamma')$ is finite, then $\centra_{G_1}(\Gamma')$  is compact (because it contains $\centra(\Gamma')$ as a lattice), hence trivial since it is normal in the group $G_1$, which  acts minimally on $\RR^n \times X'_1 \times \dots \times X'_p$. Therefore $G_1$ is discrete, and we obtain the desired splitting by setting $X_1 = \RR^n \times X'_1 \times \dots \times X'_p$ and $X_2 = X''_1 \times \dots \times X''_q$. This concludes the proof.

%

We record at this occasion that Theorem~4.2(i) \emph{does not} generalise to infinitely generated lattices. Indeed, consider the lattice $\Gamma$ in $\Isom(\RR) \times \Aut(T)$ constructed in Example~\ref{exam:RxT} above; by construction, $\Gamma$ projects discretely to $\Aut(T)$. Therefore $\Gamma$ does not satisfy the conclusion of Theorem~4.2(i) even though it is irreducible as an abstract group; the latter fact follows readily from the classification of commuting elements in free products alluded to in Example~\ref{exam:RxT}.

\item Lemma 4.7 holds under the additional assumption that $X'$ has no Euclidean factor. 

\item Theorem 4.10 holds since it concerns finitely generated residually finite lattices, in which case the Euclidean factor theorem holds by Theorem~\ref{thm:EuclideanFactor}\ref{it:ResFin} above. The references to Lemma~4.7 are harmless since we are in the setting where the additional assumption necessary for Lemma~4.7 holds. 

\item 
In Theorem~4.11, the hypothesis that $\Gamma$ and $\Lambda$ do not split virtually a $\ZZ^n$ factor should be replaced by the condition that they do not commensurate a $\ZZ^n$ subgroup. By Theorem~\ref{thm:EuclideanFactor}\ref{it:commen} this implies that the ambient \cat spaces $X$ and $Y$ do not have non-trivial Euclidean factors, and the proof goes through without changes. 

\item Corollary~4.13 holds without changes. Indeed, although its proof currently relies on Theorem~3.8, this dependence can be avoided by using Theorem~\ref{thm:EuclideanFactor}\ref{it:commen}, together with a similar argument as in the proof of Theorem~4.11.

\item Theorem~6.1 holds; the statement of part (ii) contains typos and should  be replaced by the following: \emph{There is a normal subgroup $\Gamma_D \se \Gamma$ which is either finite or infinitely generated, and such that the quotient $\Gamma/ \Gamma_D$ is an arithmetic lattice in a product of semi-simple Lie and algebraic groups.} In the proof, the definition of $\Gamma_D$ has to be slightly modified, since it can happen that $X'$ has a non-trivial Euclidean factor $\RR^n$. We write $\Isom(X') = S \times A \times D$ as in the proof of Theorem~\ref{thm:EuclideanFactor} above. Since $\Gamma$ is irreducible by hypothesis, it follows from \cite[Th.~4.2(i)]{Caprace-Monod_discrete}, together with the Borel density theorem, that the projection of $\Gamma$ to $S$ is dense. We replace  $A$ and $D$ by the closures of the respective projections of $\Gamma$. 

Let $F < \Gamma$ be a commensurated free abelian subgroup of rank $n$ provided by Theorem~\ref{thm:EuclideanFactor}\ref{it:commen}. Upon replacing $F $ by a finite index subgroup we may assume that $F$ acts properly by translations on the Euclidean factor. Let $M$ be the normal closure of $F$ in $\Gamma$, and $N$ be the kernel of the projection of $\Gamma $ to $S \times A$. Since  the projection of $F$ to $S$ is trivial, the projection of $M$ to $S$ is trivial as well. Define $\Gamma_D = M \cdot N$. 

Observe that $N$ is a discrete normal subgroup of $D$. Thus $\Gamma/N$ embeds as a lattice in $S \times A \times (D/N)$. The image $MN/N$ of $M$ now embeds in $A$ (because $M$ has trivial image in $S$), and is thus an abelian normal subgroup of $\Gamma/N$. Let $M_A$ (resp. $M_D$) denote the closure of the projection of $M$ to $A$ (resp. to $D/N$). Let $A' = A/M_A$ and $D' = (D/N)/M_D$. Then $A'$ is compact  and $D'$ is totally disconnected. We next show that the natural image of $M$, which is isomorphic to $MN/N$, is a lattice in $M_A \times M_D$. To this end, remark first that $MN/N$ is discrete. Moreover it contains the group $F$ whose projection to $M_A$ is a lattice and whose projection to $M_D$ is compact. Since the projection of $M$ to $M_D$ is dense by construction, it follows that $MN/N$ is cocompact in $M_A \times M_D$. This confirms that $MN/N$ is a lattice in $M_A \times M_D$.  We deduce  that $\Gamma/\Gamma_D \cong (\Gamma/N)/(MN/N)$ is a cocompact lattice in $S \times A' \times D'$, hence also in $S \times D'$ since $A'$ is compact. Notice moreover that the image of the projection of $\Gamma/\Gamma_D$ to $S$ is dense (because it coincides with the image of the projection of $\Gamma$). Similarly, the image of the projection of   $\Gamma/\Gamma_D$ to $D'$ is dense (by the definition of $D'$).  The conclusion of Theorem 6.1 therefore holds as a consequence of~\cite[Th.~5.18]{Caprace-Monod_discrete}.

It remains to show that $\Gamma_D$ cannot be infinite and  finitely generated. When $A$ is trivial (i.e. $X'$ has no Euclidean factor) the original argument goes through. Otherwise the group $M$ defined above is non-trivial. If $\Gamma_D$ is finitely generated, then so is $MN/N$, since it is a quotient of $\Gamma_D$. Thus the projection of $\Gamma$ to $A$ normalises a finitely generated cocompact group of translations. It follows from Lemma~\ref{lem:normaliser} that the projection of $\Gamma$ to $A$ is virtually abelian. By Theorem~\ref{thm:EuclideanFactor}\ref{it:VirtAb} this implies that  $\Gamma$ virtually decomposes as a direct product, contradicting the hypotheses. 

\item Theorems~6.2 and~6.3 hold without changes.

\item Theorem~6.6 holds without changes, since the hypotheses imply that $\Gamma$ is finitely generated and residually finite, in which case 
 the Euclidean factor theorem holds by Theorem~\ref{thm:EuclideanFactor}\ref{it:ResFin} above. 

\end{list}

\bigskip
\noindent
\textbf{Acknowledgements.}
We are grateful to Ian Leary and Ashot Minasyan for sending us their preprint~\cite{LM} and for alerting us to the inconsistency between their beautiful results and our mistaken claim. We also thank them for comments on an earlier version of this note. 



\providecommand{\bysame}{\leavevmode\hbox to3em{\hrulefill}\thinspace}

\end{document}